\documentclass[a4paper]{amsart}

\usepackage[english]{babel}
\usepackage[utf8x]{inputenc}
\usepackage[T1]{fontenc}

\usepackage[a4paper,top=3cm,bottom=2cm,left=3cm,right=3cm,marginparwidth=1.75cm]{geometry}

\usepackage{amsmath}
\usepackage{graphicx}
\usepackage[colorinlistoftodos]{todonotes}
\usepackage[colorlinks=true, allcolors=blue]{hyperref}

\theoremstyle{plain}

\newtheorem{theorem}{Theorem}[section]

\newtheorem{lemma}[theorem]{Lemma}

\newtheorem{conjecture}[theorem]{Conjecture}

\theoremstyle{definition}
\newtheorem{definition}[theorem]{Definition}

\theoremstyle{remark}
\newtheorem{example}[theorem]{Example}

\newtheorem{remark}[theorem]{Remark}

\newcommand{\re}{\mathbb R}

\newcommand{\ca}{\mathcal A}

\newcommand{\cm}{\mathcal M}
\newcommand{\cp}{\mathcal P}

\newcommand{\cq}{\mathcal Q}

\newcommand{\cC}{\mathcal C}

\newcommand{\ci}{\mathcal I}

\newcommand{\cd}{\mathcal D}

\newcommand{\mR}{\mathbb{R}}

\title{Fair representation in the intersection of two matroids}

\author{Ron Aharoni}
\address[Ron Aharoni]{Department of Mathematics\\ Technion, Haifa, Israel}
\email{raharoni@gmail.com}

\author{Eli Berger}
\address[Eli Berger]{Department of Mathematics\\ Haifa University, Haifa, Israel}
\email{berger.haifa@gmail.com}

\author{Dani Kotlar}
\address[Dani Kotlar]{Department of Computer Science\\ Tel-Hai College, Upper Galilee, Israel}
\email{dannykotlar@gmail.com}

\author{Ran Ziv}
\address[Ran Ziv]{Department of Computer Science\\ Tel-Hai College, Upper Galilee, Israel}
\email{ranzivziv@gmail.com}

\begin{document}
\maketitle

\begin{abstract}
For a simplicial complex $\cC$ denote by $\beta(\cC)$ the minimal number of edges from $\cC$ needed to cover the ground set. If $\cC$ is a matroid then for every partition $A_1, \ldots, A_m$ of the ground set there exists a set $S \in \cC$ meeting each $A_i$ in at least $\frac{|A_i|}{\beta(\cC)}$ elements. We conjecture that a slightly weaker result is true for the intersection of two matroids: if $\cd=\cp \cap \cq$, where $\cp,\cq$ are matroids on the same ground set $V$ and $\beta(\cp), \beta(\cq) \le k$,  then for every partition $A_1, \ldots, A_m$ of the ground set there exists a set $S \in \cd$ meeting each $A_i$ in at least $\frac{1}{k}|A_i|-1$ elements. We prove that if $m=2$ (meaning that the  partition is into two sets) there is a set meeting each $A_i$ in at least   $(\frac{1}{k}-\frac{1}{|V|})|A_i|-1$ elements.

\end{abstract}

\baselineskip 18pt

\section{Terminology and main theme}

A hypergraph $\cC$ is called a {\em simplicial complex} (or just a ``complex'') if it is  closed
down, namely $e \in \cC$ and $f \subseteq e$ imply $f \in \cC$. A complex is called a {\em matroid} if all its maximal sets are of the same size, and this is true also for each of its induced hypergraphs. We denote by $rank(\cC)$  the maximal size of an edge in $\cC$.
 Also  define $\zeta(\cC)=\max_{S \subseteq V(\cC)}  \frac{|S|}{rank(\cC[S])} $ (here $\cC[S]$ is the set of edges of $\cC$ contained in $S$).
\begin{definition}
An {\em edge-cover} (or plainly a {\em cover}) of a complex $\cC$ is a collection of edges whose union is $V(\cC)$. The minimal size of an edge cover is denoted by  $\beta(\cC)$, and is called the {\em edge covering number} of $\cC$. A {\em fractional cover} of the vertices of a hypergraph $\cC$ by edges is a function $f:E(\cC) \to \mR^+$ satisfying $\sum_{v\in A \in \cC}f(A) \ge 1$ for all $v\in V(\cC)$.
Let $\beta^*(\cC)=\min \sum_{A \in \cC} f(A)$, where $f$ ranges over all fractional covers of vertices by edges in  $\cC$.

\end{definition}

\begin{remark}
A  common notation  for $\beta(\cC)$ is $\rho(\cC)$, but this may generate
confusion with the notation for the rank of a matroid.
\end{remark}

Clearly, $\beta(\cC)\ge \zeta(\cC)$, and a well known theorem of Edmonds \cite{edmondsdec} is:
\begin{theorem}\label{edmonds}
If $\cp$ is a matroid then    $\beta(\cp)=\lceil\zeta(\cp)\rceil$.
\end{theorem}

The   objects studied in this paper are complexes that are the intersection of two matroids, where $\cC$  being the intersection of the matroids $\cp$ and $\cq$ means that $A \in \cC$ if and only if $A \in \cp$ and $A \in \cq$. For brevity, we call such a complex a {\em dimatroid}.

A well known  example of a dimatroid is that of the matching complex $\cm(G)$ of a bipartite graph $G$.  The ground set of this matroid is $E(G)$, and the two respective matroids are partition matroids,
one whose parts are the stars in one side of the graph, and the other having as parts the stars in the other side of the graph.
In this case a mysterious phenomenon occurs: the intersection does not affect $\beta$. Although the intersection of the two matroids is significantly poorer than each of them, its edge-covering number  is the maximum of their edge covering numbers. This is the content of
 K\"onig's  famous edge coloring theorem.
\begin{theorem}\label{konigc}
If $G$ is bipartite then  $\beta(\cm(G))=\Delta(G)$.
 \end{theorem}
 Note that $\Delta(G)$, the maximal degree of a vertex of $G$, is $\max(\beta(\cp), \beta(\cq))$, where $\cp, \cq$ are the two matroids whose intersection is $\cm(G)$. This is but one case of a probably general phenomenon: that dimatroids behave particularly well with respect to edge covers, and to representation (a notion to be expanded below).

\begin{conjecture}\cite{matcomp}\label{betaint}
If $\cp, \cq$ are matroids on the same vertex set then $$\beta(\cp \cap \cq) \le \max(\beta(\cp), \beta(\cq))+1.$$
\end{conjecture}

In fact, we do not have a counterexample to the stronger $\beta(\cp \cap \cq) \le \max(\beta(\cp), \beta(\cq)+1).$
In \cite{ak}  it was shown that
Conjecture \ref{betaint}, if true, is sharp for all values of $\max(\beta(\cp), \beta(\cq))$. A particularly simple example is the intersection of the graphic matroid $\cp$ on  $E(K_4)$ and the partition matroid $\cq$ whose three parts are the three perfect matchings in $K_4$, in which  $\beta(\cp)=\beta(\cq)=2$, while $\beta(\cp \cap \cq)=3$.
In \cite{matcomp} it was proved (using topology) that $\beta(\cp \cap \cq) \le 2\max(\beta(\cp), \beta(\cq))$.
 In \cite{abz} the conjecture was proved when $\beta(\cp)=\beta(\cq)=2$. In \cite{KZ05} it was conjectured that if no element of the ground set has $k+1$ disjoint spanning sets in either $\cp$ or $\cq$, then $\beta(\cp\cap\cq)\le k$. The conjecture was proved there for $k=2$.

 A possible strengthening  of Conjecture \ref{betaint} is that $\beta(\cp \cap \cq) \le \zeta(\cp \cap \cq)+1$.

Conjecture \ref{betaint} is reminiscent of another famous decomposition theorem - Vizing's theorem, stating that in any graph $G$ there holds
$\beta(\cm(G))\le \Delta(G)+1$ - the mysterious price of $1$ appearing yet again. The matching complex of a general graph is not a dimatroid, but something very close to it - a $2$-polymatroid (see \cite{lovaszplummer} for a definition). Clearly, in a graph $\zeta(\ci(G)) \ge \Delta(G)$, and it is an interesting question whether the weaker version of Vizing's theorem, $\beta(\ci(G)) \le \zeta(\ci(G))+1$, has a simple proof. Note that if we take a multigraph instead of a graph, we get the famous Seymour-Goldberg Conjecture \cite{seymour,goldberg}, stating that if $G$ is a multigraph then $\beta(\cm(G)) \le \beta^*(\cm(G)) +1$.

The last remark  leads to the following fractional  version  of Conjecture \ref{betaint}, that we shall use as a main tool:

\begin{theorem}\label{fracbeta}
If $\cp$ and $\cq$ are matroids on the same vertex set then $$\beta^*(\cp \cap \cq) = \max(\zeta(\cp), \zeta(\cq)).$$
\end{theorem}

A proof can be found, e.g., in  \cite{matcomp}.

\section{Fair and almost fair representation}


We say that a set $S$ represents a set $A$ $\alpha$-{\em fairly} (where $\alpha$ is a positive real number) if $|S \cap A| \ge \lfloor \alpha |A| \rfloor$, and that $S$ represents $A$ {\em almost $\alpha$-fairly} if $|S \cap A| \ge \lfloor \alpha |A|\rfloor -1$.
 A set is said to represent a partition $\ca=(A_1, \ldots ,A_m)$ of $V(\cC)$ (almost) $\alpha$-fairly if it represents all $A_i$'s (almost) $\alpha$-fairly.

 The following can be proved, e.g., by Edmonds' two matroids intersection theoerm \cite{edmondsdec} or using polyhedral methods:

\begin{theorem}\label{fairmatroid}
If $\cp$ is a matroid, then every partition of $V(\cp)$ has a fair $\frac{1}{\zeta(\cp)}$- representation by a set belonging to $\cp$.
\end{theorem}

It is possible that dimatroids behave almost as nicely with respect to fair representaion - with the fairness parameter inherited from their matroid constituents.

\begin{conjecture}\label{fairrepintmat}
Let $\cp,\cq$ be two matroids on the same ground set $V$, let $\cd=\cp \cap \cq$, and let $\zeta=\max(\zeta(\cp), \zeta(\cq))$.
Then any  partition  of  $V$  has an almost  $\frac{1}{\zeta}$-fair representation by a set from $\cd$.
\end{conjecture}

The matching complex of $C_4$, partitioned into two matchings, shows that one cannot hope for $\frac{1}{\zeta}$-fair representation, but only for almost fair representation.

 In \cite{fairrep} Conjecture \ref{fairrepintmat} was proved when  the dimatroid is the intersection of two partition matroids, each with parts of size $2$. This can be reduced to the case in which the complex is the complex of independent sets of vertices on a path.
\begin{theorem}\label{treesconj0}
If $P$ is a path and $(A_1, \ldots ,A_m)$ is a
partition of its vertex set $V$, then
there exists a subset $S$ of $V$ that is independent in $P$ and satisfies:
 $|S\cap A_i|\ge \frac{|A_i|}{2}-1$ for all $i \le m$.
\end{theorem}

In fact, something somewhat stronger is true - the same holds also when $P$ is a cycle. Another result proved in \cite{fairrep} is that  there exists an independent set $S$, such that $\sum_{i \le m}(\frac{|A_i|}{2}- |S\cap A_i|)^+ \le \frac{m}{2}$. This result is almost sharp. A conjecture  stated in \cite{fairrep} is that there is a set $S$ satisfying both conditions, namely:

\begin{conjecture}\label{treesconj0conj}
If $P$ is a path and $(A_1, \ldots ,A_m)$ is a
partition of its vertex set $V$, then
there exists a subset $S$ of $V$ that is independent in $P$ and satisfies:
 $|S\cap A_i|\ge \frac{|A_i|}{2}-1$ for all $i \le m$, with strict inequality
holding for all but at most $\frac{m}{2}$ sets $A_i$.
\end{conjecture}

Even the case of paths seems to be non-trivial: the proofs of the above theorems use Borsuk's theorem.
Another case solved in \cite{fairrep} is that of the matching complex  $\cm(K_{n,n})$, and the partition is into three parts.

In \cite{fekszabo} a particular case is studied, where the partition is into two sets,  one of the matroids, say $\cp$, is the acyclic (graphic) matroid of a graph, and the other ($\cq$) is its dual   - so their intersection is the set of all acyclic sets of edges whose complement contains a spanning tree. In \cite{fekszabo} the following was proved:

\begin{theorem}
If $G$ is a graph such that $E(G)$ can be partitioned into two spanning trees, then for every set $A \subseteq E(G)$ there exist complementary  spanning trees $S,T$ such that $$\mid |E(S) \cap A|-|E(T) \cap A| \mid \le 1.$$
\end{theorem}

In our terminology, $E(S)$ is a set that represents the partition $(A,E(G)\setminus A)$ $\frac{1}{2}$-fairly. This type of pairs of matroids yields another example showing that the ``almost'' qualification is needed in Conjecture \ref{fairrepintmat}:

\begin{example}
Let  $\cp$ be the graphic matroid on $E(K_4)$,  and let $\cq$ be its dual. Let $(A_1,A_2,A_3)$ be the partition of $E(K_4)$  into three matchings. Then $\zeta(\cp)=\zeta(\cq)=2$, and there is no set in $\cp \cap \cq$ meeting all $A_i$'s.
\end{example}

In this paper we prove a slightly weaker version of Conjecture \ref{fairrepintmat}, when  the partition is into two sets.

\begin{theorem}\label{main}
Let $\cp,\cq$ be two matroids on the same ground set $V$, and let $\cd=\cp\cap \cq$. Let $n=|S|$, $\zeta=\max(\zeta(\cp), \zeta(\cq))$ and $\delta_{\zeta,n}=\frac{1}{\zeta}-\frac{1}{n}$.
Then any  partition  of  $V$  into two sets has an almost  $\delta_{\zeta,n}$-fair representation by a set from $\cd$.
\end{theorem}

\section{Exchanges in dimatroids}

In matroids it is possible to switch between  independent sets of the same size  by sequences of exchanges. The same is true also for dimatroids. The next theorem enables the use of ``mean value'' arguments in dimatroids.

\begin{theorem}\label{moving}
Let $\cp,\cq$ be two matroids, and let $\cd=\cp \cap \cq$. Let $S,T\in\cd$ be two sets of the same size $g$. Then there exist a sequencing $s_1,\ldots,s_g$ of the elements of $S$ and a sequencing $t_1,\ldots,t_g$ of the elements of $T$ such that $S-s_1+t_1-s_2+\ldots +t_{i-1}-s_i \in \cd$ for all $i=1,\ldots,g$.

\begin{proof}
Below we use the common notation $A+x$ for $A \cup\{x\}$ and $A-a$ for $A \setminus \{a\}$.

Let $S_1=S-s_1$ for some $s_1\in S$. Since $|T|>|S_1|$, there exists $t_1\in T$ such that $S_1+t_1\in \cp$. If also $S_1+t_1\in \cq$ we continue by choosing $s_2\in S_1$ arbitrarily. Otherwise, we choose any  $s_2$  in $C_\cq(S_1,t_1)$, the circuit in $\cq$ contained in $S_1+t_1$.   and then  $S_1+t_1-s_2\in \cd$. Now, suppose we have obtained a set $S_i=S-s_1+t_1-s_2+\cdots+t_{i-1}-s_i\in\cd$. Since $|S_i|<|T|$ there is an element $t_i\in T$ such that $S_i+t_i\in\cp$. If also $S_i+t_i\in\cq$ we can pick $s_{i+1}$ arbitrarily from $S_i$ and we are done. Otherwise we pick $s_{i+1}\in C_\cq(S_i, t_i)\cap S$ (such an element must exist since $T\in\cq$) and we have $S_i+t_i-s_{i+1}\in\cd$ as required.
\end{proof}
\end{theorem}

\section{Proof of Theorem \ref{main}. }

For a complex $\cC$ and a positive integer $g$, let  $\cC^g=\{A \in \cC \mid |A| \le g\}$. Clearly, if $\cm$ is a matroid then so is $\cm^g$.

A main tool we shall use is:

\begin{lemma}\label{main:lemma}
Let $\cm$ be a matroid with $|V(\cm)|=n,\zeta(\cm)=\zeta$, and let $g$ be an integer not larger than $\frac{n}{\zeta}$.
Then $\zeta(\cm^g)=\frac{n}{g}$.
\end{lemma}

\begin{proof}
By the definition of $\zeta$,
$\zeta(\cm^g)\ge \frac{n}{rank_{\cm^g}(V)}\ge \frac{n}{g}$.

For the other direction, in order to show that
 $\zeta(\cm^g)\le \frac{n}{g}$ we have to show that for every $S\subseteq V$ it is true that  $\frac{|S|}{rank_{\cm^g}(S)}\le \frac{n}{g}$ .
That is, we want to show that  every subset $S$ of $V$ contains  a set
$R$   belonging to $\cm$ such that $r=|R|$ satisfies  $r\le g$ (so that $R \in \cm^g$) and $r \ge \frac{|S|g}{n}$.
By the definition of $\zeta$,
there exists a subset $T \in \cm$ of $S$ of   size $t\ge \frac{|S|}{\zeta}$, and since $g \le \frac{n}{\zeta}$ we have $t \ge g\frac{|S|}{n}$. Let $r=\lceil g\frac{|S|}{n} \rceil$. Since $t$ is an integer, by the above $t \ge r$, so there exists a subset $R$ of $T$ of size $r$. Since $|S| \le n$ we have $r \le g$, and thus $R$ is the desired set.
\end{proof}

We can now prove Theorem \ref{main}.

\begin{proof}
Let $n=|V|$, $g=\lfloor \frac{n}{\zeta} \rfloor$ and $h=\frac{n}{g}$.
Write $n=g\zeta+\theta$, where $\theta <\zeta$. We have
\begin{equation}\label{zeta-h}
\frac{1}{h}=\frac{g}{n}=\frac{n-\theta}{n\zeta}=\frac{1-\theta/n}{\zeta}>\frac{1}{\zeta}-\frac{1}{n}  =\delta_{\zeta,n}
\end{equation}


By Theorem  \ref{fracbeta} and Lemma~\ref{main:lemma} we have $\beta^*(\cd^g)=h$. Let $f: \cd^g \to  \re$ be a minimal fractional cover.  Then,

\begin{equation*}
n=\sum_{v \in V}1 \le \sum_{v \in V}\sum_{v\in e \in supp(f)}f(e)= \sum_{e \in supp(f)}f(e)\sum_{v \in e}1=\sum_{e \in supp(f) }f(e)|e|.
\end{equation*}

Since $\sum_{e\in supp(f)}f(e)=h$, the weighted average of the sizes of the edges in $supp(f)$ is at least $\frac{n}{h}=g$. Since $|e|\le g$ for all $e \in \cd^g$, we must have that $|e|=g$ for all $e \in \cd^g$.

Let $(A,B)$ be a partition of $V$. Clearly, some set in $supp(f)$ represents $A$ $\delta_{\zeta,n}$-fairly. Otherwise, \eqref{zeta-h} yields
\begin{equation*}
\begin{split}
|A|=&\sum_{v \in A}1 \le \sum_{v \in A}\sum_{v\in e \in supp(f)}f(e)\\
=& \sum_{e \in supp(f)}f(e)\sum_{v \in e\cap A}1=\sum_{e \in supp(f) }f(e)|e\cap A|\\
<&\sum_{e \in supp(f) }f(e)\delta_{\zeta,n}|A| < \sum_{e \in supp(f) }f(e)\frac{|A|}{h}\\=&|A|.
\end{split}
\end{equation*}

A contradiction. The same holds for $B$.

So, there exist  a set $S\in supp(f)$ that represents $A$ $\delta_{\zeta,n}$-fairly, and a set $T\in supp(f)$ that represents $B$ $\delta_{\zeta,n}$-fairly. If any of $S$ or $T$ represents both $A$ and $B$ almost $\delta_{\zeta,n}$-fairly, then we are done.
Otherwise, by Theorem \ref{moving}, there exists a sequence of sets $S=S_0, S_1, \ldots, S_g=T$ in $\cd$, such that $|S_1|=\cdots=|S_{g-1}|=g-1$ and every two adjacent sets in the sequence differ by one element.
Now, each set in this sequence must represent at least one of $|A|$ or $|B|$ almost $\delta_{\zeta,n}$-fairly. To see this, suppose $S_i$ does not represent $A$ $\delta_{\zeta,n}$-fairly. We have
\begin{equation*}
\begin{split}
|S_i\cap B|&\ge g-1-|S_i\cap A|
 > g-1-(\frac{1}{\zeta}-\frac{1}{n})|A|\\
&=g-1-(\frac{1}{\zeta}-\frac{1}{n})(n-|B|)
=g-\frac{n}{\zeta}+(\frac{1}{\zeta}-\frac{1}{n})|B|\\
&> (\frac{1}{\zeta}-\frac{1}{n})|B| - 1.
\end{split}
\end{equation*}

Thus, $S_i$ represents $B$ almost $\delta_{\zeta,n}$-fairly. Since adjacent sets in the sequence differ by one element and $S_0$ represent $A$ $\delta_{\zeta,n}$-fairly, it follows that one of the sets in this sequence must represent both $A$ and $B$ almost $\delta_{\zeta,n}$-fairly.

\end{proof}

{\bf Remark:}~~
Truncating the matroids at $g$ is intended to obtain a fractional cover whose support consists of sets of equal size. There is another natural approach to obtaining this aim, which is to take an arbitrary fractional cover, and balance its support sets. This will require proving the following conjecture, which is of  interest on its own:

\begin{conjecture}
If $\cd$ is a dimatroid and $C,D \in \cd$, then there exist $C', D' \in \cd$ of almost equal size whose union is $C \cup D$.
\end{conjecture}

\end{document}